\documentclass[english]{smfart}
\usepackage[francais,english]{babel}
\usepackage{smfthm}
\usepackage{euscript}
\usepackage{amssymb}
\usepackage{verbatim}
\usepackage{multirow}
\usepackage{xypic}

\newcommand{\Z}{{\mathbb{Z}}}
\newcommand{\Q}{{\mathbb{Q}}}
\newcommand{\C}{{\mathbb{C}}}
\newcommand{\A}{{\mathbb{A}}}
\newcommand{\G}{{\mathbb{G}}}
\newcommand{\OO}{{\scr{O}}}

\newcommand{\Sp}{{\mathrm{Sp}}}

\newcommand{\red}{{\mathrm{red}}}
\newcommand{\ord}{{\mathrm{ord}}}
\newcommand{\sqf}{{\mathrm{sf}}}
\newcommand{\GL}{{\mathrm{GL}}}
\newcommand{\Hom}{{\mathop{\mathrm{Hom}}\nolimits}}
\newcommand{\End}{{\mathop{\mathrm{End}}\nolimits}}
\newcommand{\sHom}{{\mathop{{\scr Hom}}\nolimits}}
\newcommand{\ip}[1]{\langle #1 \rangle}
\newcommand{\scr}[1]{\EuScript{#1}}
\newcommand{\Tate}{{\underline{\mathrm{Tate}}}}

\author{Nick Ramsey} \address{Department of Mathematics, DePaul
  University} \email{nramsey@depaul.edu} \title[interpolation of
  square roots]{$p$-adic interpolation of square roots of central
  $L$-values of modular forms}

\begin{document}
\frontmatter
\maketitle
\tableofcontents

\section{Introduction}

The theme of $p$-adic interpolation of special values of $L$-functions
has become a standard one in algebraic number theory.  In
\cite{serrepadic}, Serre gave a construction of the Kubota-Leopoldt
$p$-adic $L$-function by introducing $p$-adic modular forms and
studying their coefficients.  Since the work of Serre, this
relationship between $p$-adic variation of $L$-values and $p$-adic
variation of automorphic forms has become increasingly clear.  In the
case of modular forms for $\GL_2$, Coleman and Mazur (and later
Buzzard in greater generality) have constructed a rigid-analytic
\emph{eigencurve} that is the universal $p$-adic family of
(overconvergent, finite-slope) modular eigenforms (see \cite{colmaz},
\cite{buzzardeigenvarieties}).  This is the natural domain for the
study of questions related to the $p$-adic variation of these modular
forms.

Let $p$ be an odd prime and let $N$ be a positive integer that is
relatively prime to $p$.  We will denote by $D(N)$ the tame level $N$
cuspidal eigencurve.  The rigid-analytic space $D(N)$ is constructed
(as outlined in Section \ref{buzzmachine}) so that its points
parameterize finite-slope systems of eigenvalues of the Hecke
operators $T_\ell$ for $\ell\nmid Np$ and $U_p$, as well as the
diamond operators $\ip{d}_N$ for $d\in (\Z/N\Z)^\times$, acting on
cuspidal modular forms of tame level $N$.  In particular, any
classical modular form $f$ of level $Np^m$ that is an eigenform for
the Hecke operators away from the level, is an eigenform for $U_p$
with non-zero eigenvalue, and has a nebentypus character, furnishes a
point on $D(N)$.  Moreover, the collection of such \emph{classical
  points} is Zariski-dense.  By multiplicity one, we may attach to
each classical point $x\in D(N)$ a unique normalized newform
$f_0=f_0(x)$ having the system of Hecke eigenvalues prescribed by $x$.
The aim of this paper is to interpolate a square-root of a certain
ratio of twists of the central $L$-value of $f_0$ across the ``even''
part of $D(N)$.

In the spirit of simultaneous $p$-adic variation of $L$-values and
automorphic forms, a number of authors have constructed ``several
variable'' $p$-adic $L$-functions that interpolate classical critical
values of an $L$-function associated to one or more modular forms
varying in a $p$-adic family.  In particular, in the ordinary case, a
two-variable $p$-adic $L$-function interpolating the classical
critical $L$-values in a Hida family of ordinary forms has been
constructed in various places (see for example \cite{kitagawa},
\cite{hidameasure2}, \cite{greenbergstevens1}).  A similar result was
proven for Coleman families of arbitrary (fixed) finite slope by
Panchishkin in \cite{panchishkin}.  Additionally, two-variable $p$-adic
$L$-functions of at least two sorts have been constructed on the
eigencurve in \cite{emerton} and \cite{bellaiche}.

Of particular interest are the central values of the $L$-functions of
modular forms, as these are the values connected to the Birch and
Swinnerton-Dyer conjecture and its analogs in higher weight.  In many
instances, these central values (appropriately normalized) are squares
in a certain number field.  In such cases, it is natural to ask
whether a square-root can be interpolated in a $p$-adic family of
forms.  In particular, one may ask when the ``diagonal'' $L$-function
obtained by specializing a two-variable $p$-adic $L$-function to the
central point is a square.

In \cite{koblitzhi}, Koblitz studies the problem of descending
congruences between modular forms of integral weight to half-integral
weight through the Shimura lifting (a problem that he attributes to a
question of Hida - see also the work of Maeda \cite{maeda}).
Waldspurger has famously established in \cite{waldspurger} a
connection between squares of the Fourier coefficients of
half-integral weight modular forms and the central $L$-values of their
Shimura liftings.  Accordingly, Koblitz also considers, largely
conjecturally, consequences of these descended congruences for
congruences between square roots of the central $L$-values of modular
forms.  Inspired by the conjectures of Koblitz, Sofer proves in
\cite{sofer} that one can interpolate a square root of the central
$L$-value in a family of Hecke $L$-functions of imaginary quadratic
fields.  Her arguments utilize $p$-adic properties of modular forms in
an essential way.  Interpolation of square roots has also been studied
by Harris--Tilouine in \cite{harristilouine} where these authors
interpolate a central square root of a triple product $L$-function in
which one of the three modular forms varies in a Hida family of
ordinary forms.  In \cite{hida}, Hida apparently returns to his own
question and proves a $\Lambda$-adic version of Waldspurger's result
by interpolating a square-root of the ratio of the central values of
two quadratic twists of a form lying in a Hida family.  The results of
the present paper are essentially a generalization of this work of
Hida to higher-slope, placed in the context of the eigencurve.

Our interpolation takes place across the ``even'' part of the
eigencurve, that is, the union of connected components on which the
$p$-adic weight character and tame nebentypus characters are squares.
An even classical weight character is equal to $\kappa^2$,
where $$\kappa:\Z_p^\times\longrightarrow \C_p^\times$$ is of the form
$$\kappa(t)=t^{(k-1)/2}\kappa'(t)$$ for an odd positive integer $k$
and finite-order character $\kappa'$.  The tame nebentypus of such a
point is of the form $\chi^2$ for some character $\chi$.  In this
context, $\chi_0$ will denote the character $\chi_0(n)=\chi(n)\cdot
(-1/n)^{(k-1)/2}$.

For a Dirichlet character $\psi$ modulo $N$, let $D(N)_\psi$ denote
the union of connected components of the underlying reduced space
$D(N)_\red$ parameterizing forms of tame nebentypus $\psi$.  For a
square-free positive integer $n$, $\chi_n$ will denote the quadratic
character associated to the field $\Q(\sqrt{n})$.  

As is typical, we must make a choice to relate the complex-analytic
world of holomorphic modular forms and their $L$-values and the
$p$-adic world of rigid-analytic modular forms and the interpolating
functions that we construct.  Thus, we fix once and for all an
isomorphism $\C\cong \C_p$.  However, in the interest of notational
brevity, this identification is kept implicit in the arguments and
results that follow.  The reader familiar with $p$-adic interpolation
of this flavor should have no difficulty explicating the
identification should the need arise.

The following theorem is the main result of the paper.
\begin{theo}\label{maintheorem}
 Let $p$ be an odd prime and let $N$ be a square-free positive integer
 that is relatively prime to $2p$.  Suppose that the character $\psi$
 modulo $N$ is a square and let $n$ and $m$ be square-free positive
 integers with $m/n\in (\Q_\ell^\times)^2$ for all $\ell\mid 2Np$.
 There exists a character $\chi$ modulo $8N$ with $\chi^2=\psi$ such
 that on the union of irreducible components of $D(N)_\psi$ with even
 weight on which $$L(f_0\otimes
 \chi_0^{-1}\kappa'^{-1}\chi_n,(k-1)/2)$$ does not vanish generically,
 there exists a meromorphic function $\Phi_{m,n}$ such
 that $$\Phi_{m,n}(x)^2 = \frac{L(f_0\otimes
   \chi_0^{-1}\kappa'^{-1}\chi_m,(k-1)/2)} {L(f_0\otimes
   \chi_0^{-1}\kappa'^{-1}\chi_n,(k-1)/2)}\chi(m/n)
 \kappa(m/n)(m/n)^{-1/2}$$ holds at non-critical slope classical
 points $x$ outside of a discrete set.
\end{theo}
\begin{rema}\
  \begin{enumerate}
  \item We have suppressed the dependence of $f_0$, $\kappa$, and $k$
    on the point $x$ on the right side of the equation in this
    statement for brevity.
  \item 
    The value $L(f_0\otimes \chi_0^{-1}\kappa'^{-1}\chi_n,(k-1)/2)$ is
    only defined at classical points.  To say that it ``vanishes
    generically'' may be taken to mean that, among such points, it is
    supported on set of classical points that meets each irreducible
    component in a proper analytic set (in the sense of, for example,
    \cite{bgr}).  Of course, if $L(f_0\otimes
    \chi_0^{-1}\kappa'^{-1}\chi_n,(k-1)/2)$ (appropriately normalized
    so that $p$-adic interpolation is sensible) is the value of a
    globally-defined analytic function, then this condition simply
    amounts by the density of classical points to saying that this
    function vanishes identically.
  \item 
    In general, there is a sort of interplay between the choice of
    ``nebentypus square-root'' $\chi$, the square-free integer $n$ of
    the denominator, and the irreducible component on which we seek to
    interpolate.  Theorem \ref{maintheorem} is set up to interpolate
    with a fixed choice of $n$ with just one $\chi$ on as many
    components of $D(N)_\psi$ as possible.  In fact, the analysis that
    leads to the theorem allows for somewhat greater flexibility.  For
    example, if one was instead bound to a particular $\chi$ and
    working on a particular union of irreducible components of
    $D(N)_{\chi^2}$, this analysis shows that there infinitely many
    pairs $(m,n)$ for which the function $\Phi_{m,n}$ as above can be
    constructed when the denominator does not vanish generically.  On
    the other hand, if one fixes $m$ as well, then there exist
    infinitely many such $n$ if and only if there exists a single one.
  \end{enumerate}
\end{rema}

The proof of Theorem \ref{maintheorem} relies on the aforementioned
work of Waldspurger that relates the central special values of modular
$L$-functions to Fourier coefficients of half-integral weight modular
forms.  The other major input is the work of the author (\cite{mfhi},
\cite{hieigencurve}, \cite{rigidshimura}) on $p$-adic families of such
forms.  We begin by giving a detailed construction in Section
\ref{sec:nstar} of a coherent sheaf on an eigenvariety constructed
using Buzzard's eigenvariety machine (see
\cite{buzzardeigenvarieties}) that interpolates the duals of the
eigenspaces associated to the systems of eigenvalues parameterized by
that eigenvariety.  On the half-integral weight eigencurve constructed
in \cite{hieigencurve}, Fourier coefficients furnish sections of this
sheaf.  In Section \ref{sec:lindep}, we establish a general criterion
for linear dependence of sections of a coherent sheaf on a rigid
space.  When applied to a pair of Fourier coefficient sections, this
criterion implies under suitable hypotheses that these sections are
related by a meromorphic function.  Using the interpolation of the
Shimura lifting constructed in \cite{rigidshimura}, this function can
be moved to the integral weight eigencurve where it has interpolation
property stated in Theorem \ref{maintheorem}.

Generic vanishing of the $L$-value in the denominator in Theorem
\ref{maintheorem} is an obvious impediment to interpolation across an
irreducible component.  As we will see, there are other impediments
that preclude interpolation on additional components on the
half-integral weight side.  Fortunately, these will be ameliorated in
passage to integral-weight by choosing the half-integral weight
component mapping to the desired integral weight component
judiciously.  Doing so requires a detailed analysis of Waldspurger's
description of the preimage of the Shimura lifting that we carry out
in the special case of square-free tame level.

\section{Buzzard's eigenvariety machine}\label{buzzmachine}

In \cite{buzzardeigenvarieties}, Buzzard introduces a systematic way
to construct an eigenvariety given a certain system of Banach modules
and commuting endomorphisms.  The details of this construction will be
used in what follows, so we briefly recall them here.  For a more
detailed account with proofs, see \cite{buzzardeigenvarieties}.

Let $\scr{W}$ be a reduced rigid space over a complete and
discretely-valued extension $K$ of $\Q_p$.  Fix a set $\mathbf{T}$
with a distinguished element $\phi$.  Suppose that, for each admissible
open affinoid $X\subseteq \scr{W}$, we are given a Banach module $M_X$
over $\OO(X)$ satisfying a certain technical hypothesis (called
(\emph{Pr}) in \cite{buzzardeigenvarieties}) and a map
\begin{eqnarray*}
  \mathbf{T} & \longrightarrow & \End_{\OO(X)}(M_X) \\
  t & \longmapsto & t_X
\end{eqnarray*}
whose image consists of commuting endomorphisms such that $\phi_X$ is
compact for all $X$.  Suppose also that for each pair $X_1\subseteq
X_2\subseteq \scr{W}$ we are given a continuous $\OO(X_1)$-linear
injection
$$\alpha_{12}:M_{X_1}\longrightarrow M_{X_2}
\widehat{\otimes}_{\OO(X_2)}\OO(X_1)$$
that is a ``link'' in the sense
of \cite{buzzardeigenvarieties}.  Finally suppose that the links are
equivariant for the endomorphisms associated to elements of
$\mathbf{T}$ in the obvious sense and satisfy the cocycle condition
$\alpha_{13} = \alpha_{23}\circ\alpha_{12}$ for any triple
$X_1\subseteq X_2\subseteq X_3\subseteq \scr{W}$.  

Out of this data, Buzzard constructs rigid spaces $D$ and $Z$ called
the eigenvariety and spectral variety, respectively, together with
canonical maps
$$D\longrightarrow Z\longrightarrow \scr{W}.$$
The construction of $Z$
is straightforward.  For each admissible affinoid open $X\subseteq
\scr{W}$ we let $Z_X$ be the zero locus of the Fredholm determinant
$$P_X(T)=\det(1-\phi_XT\mid M_X)$$ inside $X\times \A^1$.  The links
ensure that this determinant is independent of $X$ in the sense that
if $X_1\subseteq X_2$ then $P_{X_1}(T)$ is the image of $P_{X_2}$ in
$\OO(X_1)[\![T]\!]$.  It follows that the $Z_X$ glue to a space $Z$
equipped with a canonical projection map $Z\longrightarrow \scr{W}$.  

The construction of $D$ is more difficult.  The starting point is the
following theorem proven in \cite{buzzardeigenvarieties}.
\begin{theo}\label{cover}
  Let $R$ be a reduced affinoid algebra over $K$, let $P(T)$ be a
  Fredholm series over $R$, and let $Z\subseteq \Sp(R)\times \A^1$
  denote the hypersurface cut out by $P(T)$ equipped with the
  projection $\pi: Z\longrightarrow \Sp(R)$.  Define $\scr{C}(Z)$ to
  be the collection of admissible affinoid opens $Y$ in $Z$ such that
  \begin{itemize}
  \item $Y'=\pi(Y)$ is an admissible affinoid open in $\Sp(R)$,
  \item $\pi: Y\longrightarrow Y'$ is finite, and
  \item there exists $e\in \OO(\pi^{-1}(Y'))$ such that $e^2=e$ and $Y$
  is the zero locus of $e$.
  \end{itemize}
  Then $\scr{C}(Z)$ is an admissible cover of $Z$.
\end{theo}
\noindent We will often take $Y'$ to be connected in what follows.
This is not a serious restriction, since $Y$ is the disjoint union of
the parts lying over the various connected components of $Y'$.

Let $X\subseteq \scr{W}$ and let $Y\in\scr{C}(Z_X)$ with connected
image $Y'$.  To the choice of $Y$ we can associate a factorization
$$\det(1-(\phi_X\widehat{\otimes}1)T\mid
M_X\widehat{\otimes}_{\OO(X)}\OO(Y')) = Q(T)Q'(T)$$
into relatively
prime factors with constant term $1$.  Here, $Y$ is the
zero locus of $Q$ while $Q'$ cuts out the complement of $Y$ in
$\pi^{-1}(Y')$.  The factor $Q$ is actually a polynomial of degree
equal to the degree of the finite map $Y\longrightarrow Y'$ and has a
unit for leading coefficient.  Associated to this factorization
there is a unique decomposition
$$M_X\widehat{\otimes}_{\OO(X)}\OO(Y')\cong N\oplus F$$
into closed
$\OO(Y')$-submodules $N$ and $F$ with the property that $Q^*(\phi)$
vanishes on $N$ and is invertible on $F$.  Moreover, $N$ is projective
of rank equal to the degree of $Q$ and the characteristic power
series of $\phi$ on $N$ is $Q(T)$.  

The projectors onto the submodules $N$ and $F$ lie in the closure of
$\OO(Y')[\phi]$, so it follows from the commutativity assumption that
$N$ and $F$ are preserved by the endomorphisms
$t_X\widehat{\otimes}1$.  Let $\mathbf{T}(Y)$ denote the
$\OO(Y')$-subalgebra of $\End_{\OO(Y')}(N)$ generated by these
endomorphisms.  This algebra is finite over $\OO(Y')$ and hence
affinoid, so we may define $D_Y=\Sp(\mathbf{T}(Y))$.  Since the
polynomial $Q$ is the characteristic power series of $\phi$ acting on
$N$ and has a unit for leading coefficient, $\phi$ is invertible on
$N$ and $Q(\phi^{-1})=0$ on $N$.  Thus there is a well-defined map
$D_Y\longrightarrow Y$ given by 
\begin{eqnarray*}
  \OO(Y')[T]/(Q(T)) &\longrightarrow  & \mathbf{T}(Y) \\
  T & \longmapsto & \phi^{-1}
\end{eqnarray*}
on the underlying affinoid algebras.  Now one simply glues over
varying $Y$ to obtain a space $D$ equipped with a map
$D\longrightarrow Z$.

\begin{defi}
  Let $L/K$ be a complete and discretely-valued extension of $K$.  A
  pair $(\kappa,\gamma)$ consisting of a point $\kappa\in\scr{W}(L)$
  and a map $\mathbf{T}\longrightarrow L$ is called an
  \emph{$L$-valued system of eigenvalues} of $\mathbf{T}$ acting on
  the $\{M_X\}$ if there exists an admissible affinoid open
  $X\subseteq \scr{W}$ containing $\kappa$ and a nonzero form $f\in
  M_X\otimes_{\OO(X)}L$ such that $(t_X\otimes 1)f = \gamma(t)f$ for all
  $t\in \mathbf{T}$.  The system of eigenvalues $(\kappa,\gamma)$ is
  called \emph{$\phi$-finite} if in addition $\gamma(\phi)\neq 0$.  
\end{defi}
Let $x$ be an $L$-valued point of $D$.  Let $\kappa_x$ denote the image
of $x$ in $\scr{W}$ and pick an admissible open $X\subseteq \scr{W}$
containing $\kappa$.  Then $x\in D_Y$ for some $Y\in \scr{C}(Z_X)$ so
we can associate to $x$ a map $\gamma_x:\mathbf{T}\longrightarrow L$.
The following is Lemma 5.9 of \cite{buzzardeigenvarieties}.
\begin{lemm}\label{eigenpoints}
  The association $x\longmapsto (\kappa_x,\gamma_x)$ is a well-defined
  bijection between $D(L)$ and the set of $\phi$-finite $L$-valued
  systems of eigenvalues of $\mathbf{T}$ acting on the $\{M_X\}$.  The
  image of the system of eigenvalues $(\kappa,\gamma)$ in $Z$ is
  $(\kappa,\gamma(\phi)^{-1})$.  
\end{lemm}

\section{The coherent sheaf $\scr{N}^*$}\label{sec:nstar}

Let $\{M_X\}$ be a system of Banach modules in the sense of
\cite{buzzardeigenvarieties} as above. We wish to construct a coherent
sheaf on $D$ whose fiber at a point is the linear dual of the
eigenspace for the system of eigenvalues corresponding to this point by
Lemma \ref{eigenpoints}.

Fix $X\subseteq \scr{W}$ and $Y=\Sp(A)\in \scr{C}(Z_X)$ with
connected image $Y'=\Sp(A')\subseteq X$.  Let $\mathbf{T}(Y)$ and
$N(Y)$ be as in the previous section.  We will build the sheaf
$\scr{N}^*$ by gluing the finite $\mathbf{T}(Y)$-modules
$\Hom_{A'}(N(Y),A')$.  For a general $Y\in \scr{C}(Z_X)$, we will
denote by $\mathbf{T}(Y)$ the product of the algebras obtained from
the various connected components of the image $Y'$ of $Y$ in $X$.  For
such $Y$ we will denote by $N(Y)$ the projective $\OO(Y')$-module given
by the product of the modules obtained in the previous section from
the various connected components of $Y'$.  Finally, for such $Y$,
$N(Y)^*$ will denote the $\OO(Y')$-linear dual of $N(Y)$, or
equivalently the product of the duals of the projective
modules corresponding to the various connected components of $Y'$.

\begin{lemm}\label{glueing}
  Let $Y_1,Y_2\in\scr{C}(Z_X)$.  Then $Y=Y_1\cap Y_2\in\scr{C}(Z_X)$
  and there is a canonical identification
  $$N(Y)^*\cong N(Y_i)^*\otimes_{\mathbf{T}(Y_i)} \mathbf{T}(Y).$$
\end{lemm}
\begin{proof}
  Let $A_i$ denote the affinoid algebra of $Y_i$ or $i=1,2$ and $A$
  denote that of $Y_1\cap Y_2$.  Lemma 5.2 of
  \cite{buzzardeigenvarieties} gives the first assertion as well as
  the fact that restriction induces an
  isomorphism $$\mathbf{T}(Y_i)\otimes_{A_i}A
  \stackrel{\sim}{\longrightarrow} \mathbf{T}(Y).$$
  Thus, it suffices to construct an identification 
\begin{equation}\label{ident2}
  N(Y)^*\cong
  N(Y_i)^*\otimes_{A_i}A.
\end{equation}

Each object in this purported isomorphism of $A$-algebras breaks up
into factors associated to the various connected components of the
images of $Y_i$ and $Y$ in $X$.  Let $Y'$ be a connected component
of the image of $Y=Y_1\cap Y_2$, and let $Y_i'$ denote the unique
connected component of the image of $Y_i$ containing $Y'$.  Now
re-define $Y_i$ to be its part lying over $Y_i'$ for $i=1,2$.  Each
$Y_i$ lies in $\scr{C}(Z_X)$ with connected image $Y_i'$ so the
preimages $Y_i\cap \pi^{-1}(Y') = Y_i\times_{Y_i'}Y'$ lie in
$\scr{C}(Z_X)$ with image $Y'$.  Now re-define $Y$ to be the
intersection of these preimages, so that $Y$ is also in
$\scr{C}(Z_X)$ with image $Y'$, as it is the intersection of two
unions of connected components of $\pi^{-1}(Y')$.

Thus it suffices to prove (\ref{ident2}) with these new $Y_i$ and
$Y$ where now each image $Y_i'$ and $Y'$ is connected.  In the
interest of notational brevity we adopt the convention that objects
associated with $Y_i$ carry a subscript $i$ and objects associated
with $Y$ carry no subscript at all.  Also, we will let $A_i$ denote
the affinoid algebra of $Y_i$, let $A$ denote that of $Y$, let
$A_i'$ denote that of $Y_i'$, and let $A'$ denote that of $Y'$.

Note that
\begin{eqnarray}\label{tensormess}
  N_i^*\otimes_{A_i}A &\cong &
  (N_i^*\otimes_{A_i}(A_i\otimes_{A_i'}A'))\otimes_{A_i\otimes_{A_i'}A'}A
  \\\nonumber  &\cong& (N_i^*\otimes_{A_i'}A')\otimes_{A_i\otimes_{A_i'}A'}A
\end{eqnarray}
and since $N_i$ is projective over $A_i'$ we
have 
\begin{equation}\label{dualhom}
N_i^*\otimes_{A_i'}A' = \Hom_{A_i'}(N_i,A_i')\otimes_{A_i'}A'
\cong \Hom_{A'}(N_i\otimes_{A_i'}A',A')
\end{equation}

  By definition, we have isomorphisms
  \begin{equation}\label{dumbisom}
    M\cong M_i\widehat{\otimes}_{A_i'}A'
  \end{equation}
  that are equivariant with respect to the Hecke actions on both
  sides.  Tensoring the decomposition $M_i\cong N_i\oplus F_i$ with
  $A'$ gives a decomposition 
  \begin{equation}\label{tensoreddecomp}
    M_i\widehat{\otimes}_{A_i'}A' \cong
    (N_i\otimes_{A_i'}A')\oplus (F_i\widehat{\otimes}_{A_i'}A')
  \end{equation}
  in which the first summand is the kernel of $Q_i^*(\phi)$.  Note
  that we are abusing notation slightly here by identifying $Q_i$ with
  its image in $A'[T]$. This sort of abuse will persist throughout the
  argument.  The equivariance of (\ref{dumbisom}) identifies $N$ with
  the submodule of (\ref{tensoreddecomp}) on which $Q^*(\phi)=0$.
  Since $Y$ is a union of connected components of
  $Y_i\times_{Y_i'}Y'$, there is a factorization $Q_i =
  Q\widetilde{Q}_i$ into relatively prime factors in $A'[T]$ with
  constant term $1$.  In particular, (\ref{dumbisom}) identifies $N$
  with the submodule of $N_i\otimes_{A_i'}A'$ on which $Q^*(\phi)=0$
  since $Q^*|Q_i^*$.
  
  This factorization of $Q_i$ induces a canonical decomposition
  $$A_i\otimes_{A_i'}A'\cong A'[T]/(Q_i(T))\cong A'[T]/(Q(T))\times
  A'[T]/(\widetilde{Q}_i(T)) \cong A\times \tilde{A}_i$$
  where
  $\tilde{A}_i\cong A'[T]/(\widetilde{Q}_i(T))$ is the affinoid
  algebra of $Y_i\times_{Y_i'}Y' \setminus Y$.  This in turn
  decomposes the $A_i\otimes_{A_i'}A'$-module $N_i\otimes_{A_i'}A'$
  into the direct sum of the kernels of $Q(\phi^{-1})$ and
  $\widetilde{Q}_i(\phi^{-1})$ (equivalently, the kernels of
  $Q^*(\phi)$ and $\widetilde{Q}^*(\phi)$).  As the former is
  naturally identified with $N$ via the isomorphism (\ref{dumbisom}),
  we have a decomposition
  $$N_i\otimes_{A_i'}A'\cong N\oplus \tilde{N}_i$$ which, with
  (\ref{dualhom}), gives $$N_i^*\otimes_{A_i'}A'\cong N^*\oplus
  \widetilde{N}_i^*.$$ Combining this with (\ref{tensormess}) yields an
  isomorphism
  $$N_i^*\otimes_{A_i}A\cong (N^*\otimes_{A_i\otimes_{A_i'}A'}A)\oplus
  (\widetilde{N}^*_i\otimes_{A_i\otimes_{A_i'}A'}A).$$
  As $A$ acts as
  zero on $\widetilde{N}_i$, the second summand vanishes.  By contrast,
  the action of $A_i\otimes_{A_i'}A'$ on $N$ factors through $A$, so
  the first factor is simply identified with $N^*$, and we have
  exhibited a canonical isomorphism $N_i^*\otimes_{A_i}A\cong N^*$.
\end{proof}

The isomorphisms exhibited in this proof satisfy the cocycle condition
because those of Lemma 5.2 of \cite{buzzardeigenvarieties} do (as
observed in the comments following that lemma), as do the isomorphisms
(\ref{dumbisom}) for trivial reasons.  Thus we may glue the $N(Y)^*$
for varying $Y\in \scr{C}(Z_X)$ to obtain a coherent sheaf on $D(X)$.
A similar argument to that in Lemma \ref{glueing} using the links in
place of the trivial isomorphisms (\ref{dumbisom}) and Lemma 5.6 of
\cite{buzzardeigenvarieties} in place of Lemma 5.2 of
\cite{buzzardeigenvarieties} allows us to glue these sheaves for
varying $X\subseteq\scr{W}$ to obtain a coherent sheaf $\scr{N}^*$ on
all of $D$.

Let us now compute the completed stalks and fibers of $\scr{N}^*$.
Fix an admissible open $X\subseteq \scr{W}$ and
$Y=\Sp(A)\in\scr{C}(Z_X)$ with connected image $Y'=\Sp(A')\subseteq
X$.  In the interest of notational brevity we will drop the $Y$ from
the notation of the module $N$ and the algebra $\mathbf{T}(Y)$ and
refer to them simply as $N$ and $\mathbf{T}$.  The conflict in
notation with the \emph{set} $\mathbf{T}$ should cause no confusion.
For each $y\in Y'$ we have
\begin{equation}\label{tdecomp}
  \mathbf{T}\otimes_{A'}\widehat{A}'_y
  \cong \prod_x \widehat{\mathbf{T}}_x
\end{equation}
where the product is taken over the fiber of the finite map
$D_Y\longrightarrow Y'$ over $y$.  The module
$N\otimes_{A'}\widehat{A}'_y$ is finite and projective over
$\widehat{A}'_y$, and therefore free.  This module moreover carries a
faithful action of $\mathbf{T}\otimes_{A'}\widehat{A}'_y$ since
$A'\longrightarrow \widehat{A}'_y$ is flat, and therefore breaks up as 
\begin{equation}\label{ndecomp}
  N\otimes_{A'}\widehat{A}'_y\cong \prod_x
  N\otimes_{\mathbf{T}}\widehat{\mathbf{T}}_x
\end{equation}
and each factor of (\ref{tdecomp}) acts faithfully on the
corresponding factor of (\ref{ndecomp}). 
It follows that 
$$\Hom_{A'}(N,A')\otimes_{A'}\widehat{A}'_y \cong \prod_x
\Hom_{\widehat{A}'_y}(N\otimes_{\mathbf{T}}
\widehat{\mathbf{T}}_x,\widehat{A}'_y).$$ Extending scalars of either
of these modules from $\mathbf{T}\otimes_{A'}\widehat{A}'_y$ to
$\widehat{\mathbf{T}}_x$ for a particular $x$ simply picks out the
factor corresponding to $x$, so we have canonical identifications
\begin{eqnarray*}
\Hom_{A'}(N,A')\otimes_{\mathbf{T}}\widehat{\mathbf{T}}_x &\cong& 
(\Hom_{A'}(N,A')\otimes_{\mathbf{T}}(\mathbf{T}\otimes_{A'}\widehat{A}'_y))
\otimes_{\mathbf{T}\otimes_{A'}\widehat{A}'_y}\widehat{\mathbf{T}}_x
\\ &\cong&
(\Hom_{A'}(N,A')\otimes_{A'}\widehat{A}'_y)\otimes_{\mathbf{T}\otimes_{A'}
  \widehat{A}'_y}\widehat{\mathbf{T}}_x 
\\ &\cong& \Hom_{\widehat{A}'_y}(N\otimes_{\mathbf{T}}
\widehat{\mathbf{T}}_x,\widehat{A}'_y),
\end{eqnarray*}
which completes the description of the completed stalk of $\scr{N}^*$ at
$x$.  

We now turn to the fiber.  Again let $y\in Y'$ and let $L/K$ be a
finite extension containing the residue field of $y$.  Further
extending scalars to $L$ in (\ref{tdecomp}) and (\ref{ndecomp}) we
arrive at
\begin{equation}\label{tdecompfiber}
\mathbf{T}\otimes_{A'}L \cong \prod_x
\widehat{\mathbf{T}}_x\otimes_{\widehat{A}'_y}L
\end{equation}
and 
\begin{equation}\label{ndecompfiber}
N\otimes_{A'}L
\cong \prod_x
(N\otimes_{\mathbf{T}}\widehat{\mathbf{T}}_x)\otimes_{\widehat{A}'_y}L
\end{equation}
If we fix an $L$-valued point $\lambda:\mathbf{T}\longrightarrow L$ of
$D_Y$ corresponding to a point $x$ in the fiber over $y$ (perhaps
after enlarging $L$), then the image of this point in $Y'$ gives a map
$\widehat{A}'_y\longrightarrow L$ and we may decompose as above.  The
following lemma characterizes the factor corresponding to $x$ in terms
of $\lambda$.
\begin{lemm}\label{nilpotent}
  There exists a positive integer $e$ such that $(t-\lambda(t))^e=0$
  on the factor corresponding to $x$ in (\ref{tdecompfiber}).
  Moreover, for $x'\neq x$ in the fiber over $y$, there exists $t\in
  \mathbf{T}$ such that $t-\lambda(t)$ is invertible on the factor
  corresponding to $x'$.  
\end{lemm}
\begin{proof}
  Denote the maximal ideals of $\widehat{A}'_y$ and
  $\widehat{\mathbf{T}}_x$ by $m_y$ and $m_x$, respectively.  Since
  $\widehat{\mathbf{T}}_x/m_y\widehat{\mathbf{T}}_x$ is
  finite-dimensional Noetherian local algebra over the field
  $\widehat{A}'_y/m_y$, the Krull intersection theorem implies that
  there exists an integer $e$ such that $m_x^e$ vanishes in this
  algebra.  That is to say, $m_x^e\subseteq
  m_y\widehat{\mathbf{T}}_x$.  Since $t-\lambda(t)\in m_x$ for all $t$,
  the first claim follows immediately.
  
  Now pick $x'\neq x$ in the fiber over $y$.  Pick a finite
  extension $L'/L$ containing the residue field of the local algebra
  $\widehat{\mathbf{T}}_{x'}\otimes_{\widehat{A}'_y}L$ and let
  $$\mu:\widehat{\mathbf{T}}_{x'}\otimes_{\widehat{A}'_y}L\longrightarrow
  L'$$ realize this inclusion.  Suppose that, for all
  $t\in\mathbf{T}$, we have $\mu(t-\lambda(t))=0$.  Then
  $\mu(t)=\lambda(t)$ for all $t$ and it follows from Lemma
  \ref{eigenpoints} that $x'=x$.  Thus there exists $t$ such that
  $\mu(t-\lambda(t))\neq 0$, which implies that $t-\lambda(t)$ is a
  invertible element of the local algebra
  $\widehat{\mathbf{T}}_x\otimes_{\widehat{A}'_y}L$.
\end{proof}

Let $x\in D$ and let $\lambda$ be the corresponding system of
eigenvalues.  Pick $X\subseteq \scr{W}$ containing the image of
$x$ and let
$$M_\lambda = \{ f\in M\widehat{\otimes}_{\OO(X)}L\ |\ (t\otimes 1)f =
\lambda(t)f\ \mbox{for all}\ t\in\mathbf{T}\}$$ be the corresponding
eigenspace.  Note that this space is independent of the choice of $X$
in the sense that the links that are part of the data of the system
$\{M_X\}$ identify the eigenspace obtained from any two choices of
$X$ containing the image of $x$.  Note also that $x\in D_Y$ implies
that $M_\lambda$ in fact lies in the summand $N\otimes_{A'}L$ of
$M\widehat{\otimes}_{\OO(X)}L$.  The following is an immediate
consequence of Lemma \ref{nilpotent}.

\begin{coro}\label{eigenspaceinsummand}
  In the decomposition (\ref{ndecompfiber}), the eigenspace
  $M_\lambda$ is contained in the summand
  $(N\otimes_{\mathbf{T}}\widehat{\mathbf{T}}_x)\otimes_{\widehat{A}'_y}L$
  of $N\otimes_{A'}L$.
\end{coro}

Using the above description of the completed stalk of $\scr{N}^*$ at
$x$, we can compute the fiber by extending scalars via
$\widehat{\mathbf{T}}_x\longrightarrow L$ to get
\begin{eqnarray*}
\lefteqn{
\Hom_{\widehat{A}'_y}(N\otimes_{\mathbf{T}}\widehat{\mathbf{T}}_x,
\widehat{A}'_y)
\otimes_{\widehat{\mathbf{T}}_x}L } && \\ &\cong &
(\Hom_{\widehat{A}'_y}(N\otimes_{\mathbf{T}}\widehat{\mathbf{T}}_x,
\widehat{A}'_y)\otimes_{\widehat{\mathbf{T}}_x}(\widehat{\mathbf{T}}_x
\otimes_{\widehat{A}'_y}L))\otimes_{\widehat{\mathbf{T}}_x
  \otimes_{\widehat{A}'_y}L}L\\ & \cong&
(\Hom_{\widehat{A}'_y}(N\otimes_{\mathbf{T}}\widehat{\mathbf{T}}_x,
\widehat{A}'_y)\otimes_{\widehat{A}'_y}L)\otimes_{\widehat{\mathbf{T}}_x
  \otimes_{\widehat{A}'_y}L}L \\ &\cong &
\Hom_L((N\otimes_{\mathbf{T}}\widehat{\mathbf{T}}_x)
\otimes_{\widehat{A}'_y}L,L)\otimes_{\widehat{\mathbf{T}}_x
  \otimes_{\widehat{A}'_y}L}L 
\end{eqnarray*}
where the last isomorphism follows because
$N\otimes_{\mathbf{T}}\widehat{\mathbf{T}}_x$ is free of finite rank
over $\widehat{A}'_y$.  Now by Corollary \ref{eigenspaceinsummand},
$M_\lambda$ is contained in
$(N\otimes_\mathbf{T}\widehat{\mathbf{T}}_x)\otimes_{\widehat{A}'_y}L$, and
the following lemma shows that the restriction map induces an isomorphism
$$\Hom_L((N\otimes_{\mathbf{T}}\widehat{\mathbf{T}}_x)
\otimes_{\widehat{A}'_y}L),L)\otimes_{\widehat{\mathbf{T}}_x
  \otimes_{\widehat{A}'_y}L}L \stackrel{\sim}{\longrightarrow}
\Hom_L(M_\lambda,L).$$ 

\begin{lemm}
  Let $V$ be a finite-dimensional vector space over a field $L$
  equipped with the action of a commutative $L$-algebra $\scr{A}$ and let
  $\lambda:\scr{A}\longrightarrow L$ be an $L$-algebra homomorphism.
  Then the restriction map induces an isomorphism
  $$V^*\otimes_{\scr{A}}L\longrightarrow 
  (V_\lambda)^*$$ where $V_\lambda$ is the $\lambda$-eigenspace of $V$
  for the action of $\scr{A}$.  
\end{lemm}
\begin{proof}
  First observe that the indicated map is well-defined since by
  definition $\scr{A}$ acts on $V_\lambda$ via $\lambda$.  It suffices
  to check that the dual map $$V_\lambda\longrightarrow
  (V^*\otimes_{\scr{A}}L)^*$$ is an isomorphism.  The surjection
  $$\scr{A}\longrightarrow L$$ induces a surjection
  $$V^*\longrightarrow V^*\otimes_{\scr{A}}L.$$ Dualizing we arrive at
  an injection $$(V^*\otimes_{\scr{A}}L)^*\longrightarrow V.$$ It is
  clear that this injection has image in $V_\lambda$ and it is easy to
  check that the map so obtained is the inverse of the dual of the map
  in the statement.
\end{proof}

The upshot is that we have constructed an identification
\begin{equation}\label{fibers}
\scr{N}^*(x) = \scr{N}^*\otimes_{\OO_D}L \cong \Hom_L(M_\lambda,L).
\end{equation}

\section{Comments on the Shimura lifting and irreducible components}\label{sec:shimura}

In this section, we recall the main constructions and results of the
author's papers \cite{hieigencurve} and \cite{rigidshimura}.  Fix an
odd prime $p$ and a positive integer $N$ that is relatively prime to
$p$.  Let $D(N)$ denote the tame level $N$ eigencurve parameterizing
finite-slope systems of eigenvalues of the Hecke operators $T_\ell$
for $\ell\nmid Np$, $U_p$, and the diamond operators $\ip{d}_{N}$ for
$d\in (\Z/N\Z)^\times$, acting on cuspidal modular forms of integral
weight and tame level $N$.  Similarly, let $\widetilde{D}(4N)$ denote
the eigencurve parameterizing finite-slope systems of eigenvalues of
the Hecke operators $T_{\ell^2}$ for $\ell\nmid 2Np$, $U_{p^2}$, and
the diamond operators $\ip{d}_{4N}$ for $d\in (\Z/4N\Z)^\times$,
acting on cuspidal modular forms of half-integral weight and tame
level $4N$.  The constructions of these objects is a straightforward
application of the machinery in Section \ref{buzzmachine} once the
spaces of forms and operators are defined.  The reader can consult
\cite{buzzardeigenvarieties} and \cite{hieigencurve} for the details.

\begin{rema}
  Strictly speaking, the constructions of \cite{hieigencurve} and
  \cite{rigidshimura} outlined below were carried out in the case
  where Hecke eigenconditions are imposed at \emph{all} primes.
  However, the constructions contained in these papers adapt very
  easily to the present situation where such conditions are imposed
  only for $\ell\nmid 2N$.
\end{rema}

In \cite{rigidshimura}, the author constructs a map $\mathrm{Sh}$ on
underlying reduced rigid spaces fitting into a diagram
$$\xymatrix{ \widetilde{D}(4N)_\red \ar[r]^{\mathrm{Sh}}\ar[d] &
  D(2N)_\red\ar[d] \\ \scr{W}\ar[r]^{\kappa\mapsto \kappa^2} &
  \scr{W}}$$ that interpolates the classical Shimura lifting at
classical points in the sense that it preserves the system of
eigenvalues of Hecke and diamond operators in the evident sense.  This
map takes each irreducible component of $\widetilde{D}(4N)_\red$
isomorphically onto an irreducible component of $D(2N)_\red$.

Recall from \cite{hieigencurve} that the weight character book-keeping
in half-integral weight is such that classical weight $k/2$
corresponds to the weight character $\kappa(t)=t^{(k-1)/2}$, and that
the $p$-part of the nebentypus character of a classical form is
packaged as part of the weight character.  By a \emph{classical
  weight} we shall mean one of the form $\kappa(t) =
t^{(k-1)/2}\kappa'(t)$ where $k$ is an odd positive integer and
$\kappa'$ is of finite order.  A point $x\in\widetilde{D}(4N)$ with
corresponding system of eigenvalues $\lambda_x$ will be called
\emph{classical} if there exists a classical form having this system
of eigenvalues.  The point $x$ will be called \emph{strictly
  classical} if the entire $\lambda_x$-eigenspace consists of
classical forms.  Finally, a point $x$ lying over a classical weight
$\kappa(t)=t^{(k-1)/2}\kappa'(t)$ is called \emph{low-slope} if
$v(\lambda_x(U_{p^2}))< k-2$.  In \cite{hieigencurve} it is proven
that the low-slope points are strictly classical, and in
\cite{rigidshimura} it is shown that they comprise a Zariski-dense set
in $\widetilde{D}(4N)$.  Note that the notion of low-slope is
compatible with that used in the integral weight setting in the sense
that the Shimura lifting takes weight $k/2$ to weight $k-1$.

The system of eigenvalues given by a classical point in $D(N)$
corresponds to a unique normalized newform $f_0$ that generates an
automorphic representation $\pi = \otimes_v \pi_v$ of
$\mathrm{GL}_2(\A_\Q)$.  Given a property of this representation
(e.g. ``$\pi_\ell$ is Steinberg'') we will say that a point on $D(N)$
has this property if the corresponding representation does.
Similarly, given a point of $\widetilde{D}(4N)$, we will say that it
has such a property if its image under the map $\mathrm{Sh}$ does.

We now make a few observations about how some of these properties sit
on the eigencurve $D(N)$ (and hence on $\widetilde{D}(4N)$).  A more
comprehensive study of this topic can be found in the papers
\cite{paulin1} and \cite{paulin2} of Paulin.  We content ourselves
here with a some self-contained observations at square-free tame level
that will suffice for our purposes.

First, we look at $\ell=p$.  By the finite-slope condition, a
classical point in $D(N)$ is either irreducible principal series or
(an unramified twist of) Steinberg at $p$.  The Steinberg points with
weight $k-1$ satisfy $$\lambda_x(U_p)^2 =
\chi(p)^2p^{k-3},$$ where $\chi$ is the tame nebentypus of the point.
In particular, these points lie in the preimage of the discrete
analytic set $\mu_{N}(\C_p)\cdot p^\Z$ under the analytic function
\begin{eqnarray*}
  D(N) & \longrightarrow & \G_m \\
  x & \longmapsto & \lambda_x(U_p)^2
\end{eqnarray*}
The Steinberg at $p$ locus is accordingly analytic in $D(N)$
and cannot contain an irreducible component, and thus meets each
irreducible component in a proper analytic set.

Now suppose that $N$ is square-free and $\ell\mid N$.  The values of
$\lambda_x(\ip{d}_{N})$ are constant on connected components of
$D(N)$, so the tame nebentypus character is as well.  Let $x\in D(N)$
be a point with weight $\kappa(t)=t^{k-1}\kappa'(t)$ (with $\kappa'$
of finite order) and tame nebentypus character $\chi$.  Since $N$ is
square-free, a classical point $x\in D(N)$ is either irreducible
principal series or an unramified twist of Steinberg at $\ell$.  The
ramified irreducible principal series among these points are precisely
those lying on connected components with non-trivial nebentypus at
$\ell$.

Let us restrict now to the connected components with trivial
nebentypus at $\ell$.  We wish to determine the nature of the locus in
$D(N)$ that is Steinberg at $\ell$.  To do this we first consider the
auxiliary eigencurve $D(N)^\ell$ obtained by further imposing an
eigenvalue for $U_\ell$ (that is, by adding $U_\ell$ to the set
$\mathbf{T}$ used in the construction of $D(N)$).  Let $x\in
D(N)^\ell$ be a classical point of weight $\kappa(t) =
t^{k-1}\kappa'(t)$ and tame nebentypus $\chi$.  Then $x$ is in the
Steinberg locus if an only if $$\lambda_x(U_\ell)^2 =
\kappa'(\ell)\chi(\ell)\ell^{k-3} = \chi(\ell)\ell^{-2}\kappa(\ell),$$
and this equation defines an analytic set in $D(N)^\ell$.  We conclude
that an irreducible component of $D(N)^\ell$ on which $\chi$ is
trivial at $\ell$ is either generically Steinberg at $\ell$ or
generically unramified principal series at $\ell$.  Moreover, a
component that is generically Steinberg is in fact everywhere
Steinberg, while a generically unramified principal series component
may have a proper analytic subset of Steinberg points.  Now,
exploiting the evident finite map $D(N)^\ell\longrightarrow D(N)$, we
may draw these very same conclusions for $D(N)$.

Assume now that $N$ is square-free and odd.  Using the Shimura
lifting, we get a similar classification of the components of
$\widetilde{D}(4N)$.  Here, however, we can further divide the
Steinberg components into two types.  Let us again adjoin $\ell$ to
the list of eigenconditions (via the operator $U_{\ell^2}$) to obtain
the curve $\widetilde{D}(4N)^\ell$.  Let $x$ denote a classical point
of $\widetilde{D}(4N)^\ell$ with weight $\kappa(t) =
t^{(k-1)/2}\kappa'(t)$ and tame nebentypus $\chi$.  Factor $\chi =
\chi^{(\ell)}\chi^{(\sim\ell)}$ into a character modulo a power of
$\ell$ and a character of conductor prime to $\ell$.  In particular,
if $x$ lies in a Steinberg component, then $\chi^{(\ell)}$ is
quadratic, and the Steinberg condition can be
written $$\lambda_x(U_{\ell^2})^2 =
(\chi^{(\sim\ell)}(\ell))^2\ell^{-2}\kappa(\ell)^2$$ so
that $$\lambda_x(U_{\ell^2})=\pm
\chi^{(\sim\ell)}(\ell)\ell^{-1}\kappa(\ell).$$ Each sign defines an
analytic set in $\widetilde{D}(4N)^\ell$, so a Steinberg component
here can be labeled as either a ``$+$'' component or a ``$-$''
component.  Once again, we can exploit the finite map
$\widetilde{D}(4N)^\ell\longrightarrow \widetilde{D}(4N)$ to get such
a labeling of the components of $\widetilde{D}(4N)$.

Owing to the need to have a component on the half-integral weight side
with tame nebentypus of conductor divisible by $8$, we will later work
on the eigencurve $\widetilde{D}(8N)$ with $N$ odd and square-free.
Here, the Shimura lifting maps to $D(4N)$, which contains $D(2N)$ as a
union of irreducible components along with some other
(e.g. supercuspidal at $2$) components.  However, will only be
interested in those components that lie in $D(2N)$, and we label an
irreducible component of $\widetilde{D}(8N)$ that maps under the
Shimura lifting to an irreducible component of $D(2N)$ in the manner
described above.

\section{A criterion for linear dependence of sections of a coherent module}\label{sec:lindep}

We denote the sheaf of meromorphic functions on a rigid space $X$ by
$\scr{M}_X$. This is the localization of $\OO_X$ at the subsheaf of
regular (non-zerodivisor) sections.  For a coherent module $\scr{F}$
on $X$, we denote its $\OO_X$-dual by $$\scr{F}^{\vee} =
\sHom_{\OO_X}(\scr{F},\OO_X).$$ This module is also coherent and there
is a canonical map $\scr{F}\longmapsto \scr{F}^{\vee\vee}$ as usual.

The proofs of the following two basic lemmas concerning meromorphic
functions can be found in the discussion in Section 2 of
\cite{conradmoish}.
\begin{lemm}\label{bclemm1}
  Suppose that $\Sp(A)\subseteq X$ is an admissible affinoid open.
  The choice of a finite $A$-module $M$ and an identification
  $\scr{F}|_{\Sp(A)}=\widetilde{M}$ induce an
  identification $$\Gamma(X,\scr{F}\otimes_{\OO_X}\scr{M}_X)\cong
  M\otimes_{A}\mathrm{Frac}(A).$$
\end{lemm}

\begin{lemm}\label{bclemm2}
  Let $X$ be a reduced rigid space and let
  $\pi:\widetilde{X}\longrightarrow X$ be the normalization of $X$.  For
  coherent sheaves $\scr{F}$ on $X$, there is a natural
  identification $$\scr{F}\otimes_{\OO_X}\scr{M}_X
  \stackrel{\sim}{\longrightarrow}
  \pi_*(\pi^*\scr{F}\otimes_{\OO_{\widetilde{X}}}\scr{M}_{\widetilde{X}})$$
  that is functorial in $\scr{F}$.
\end{lemm}

The next Lemma gives several equivalent characterizations of ``generic
vanishing'' of a section of a coherent sheaf on a reduced rigid space.

\begin{lemm}\label{degenlemm}
  Let $X$ be a reduced rigid space and let $\scr{F}$ be a
  coherent sheaf on $X$.  For a section $f\in \Gamma(X,\scr{F})$, the
  following are equivalent:
  \begin{enumerate}
  \item[(a)] $f$ lies in the kernel of the natural map
    $$\scr{F}\longrightarrow \scr{F}\otimes_{\OO_X}\scr{M}_X;$$
  \item[(b)] $f$ lies in the kernel of the natural map 
    $$\scr{F}\longrightarrow \scr{F}^{\vee\vee};$$
  \item[(c)] $f$ vanishes on a nonempty admissible open in $X$ that meets
    each irreducible component of $X$;
  \item[(d)] $f$ is supported on a nowhere-dense analytic set in $X$.
  \end{enumerate}
\end{lemm}
\begin{proof}
  The fact that (a) and (b) are equivalent follows from the facts that the 
  right and bottom arrows in the natural diagram 
 $$\xymatrix{ \scr{F}\ar[r]\ar[d] & \scr{F}^{\vee\vee}\ar[d]
    \\ \scr{F}\otimes_{\OO_X} \scr{M}_X \ar[r] &
    \scr{F}^{\vee\vee}\otimes_{\OO_X}\scr{M}_X}$$ are injections.
  Suppose that a section $f$ satisfies condition (a).  Then there
  exists a non-zerodivisor $a$ on $X$ with $af=0$.  It follows that
  the support of $f$ lies in the zero locus of $a$, which is a
  nowhere-dense analytic set in $X$ since $a$ is not a zero-divisor,
  so (a) implies (d).  That (d) implies (c) follows because the
  complement of an analytic set is an admissible open.
  
  We claim that (c) implies (a), which will complete the proof.  By
  Lemma \ref{bclemm2}, it suffices to prove this after pulling back to
  the normalization.  Looking at connected components, we see that it
  suffices to prove the claim when $X$ is moreover normal and
  connected.  But then any two points in $X$ can be connected by a
  finite chain of nonempty admissible affinoid opens, so it suffices
  to check that for an inclusion $\Sp(B)\subseteq \Sp(A)$ of such
  opens, the restriction
  map $$\Gamma(\Sp(A),\scr{F}\otimes_{\OO_X}\scr{M}_X) \longrightarrow
  \Gamma(\Sp(B),\scr{F}\otimes_{\OO_X}\scr{M}_X)$$ is injective.  By
  normality, $A$ and $B$ are domains, and this injectivity follows
  readily from the description of sections in Lemma \ref{bclemm1}.
\end{proof}

Recall that a subset $S$ of a rigid space $X$ is \emph{Zariski-dense}
if the only analytic set containing $S$ is all of $X$.
\begin{lemm}\label{mapsto0}
  Let $X$ be a normal and connected rigid-analytic curve,
  let $\scr{F}$ be a coherent sheaf on $X$, and let $s$ be a global
  section of $\scr{F}$.  If $s$ vanishes at a Zariski-dense set in
  $X$, then the image of $s$ in $\scr{F}\longrightarrow
  \scr{F}\otimes_{\OO_X}\scr{M}_X$ is zero.
\end{lemm}
\begin{proof}
With the hypotheses on $X$, the sheaf $\scr{F}^{\vee\vee}$ is
locally-free since it is torsion-free, so the zero-locus of a section
is an analytic set.  As the image of $s$ in this sheaf vanishes on a
Zariski-dense set, this images vanishes and the result follows from
Lemma \ref{degenlemm}.
\end{proof}

\begin{lemm}\label{lindep}
  Let $X$ be a reduced rigid-analytic space of pure dimension $1$, let
  $\scr{F}$ be a coherent sheaf on $X$, and let $s_1$ and $s_2$ be global
  sections of $\scr{F}$ such that
  \begin{enumerate}
  \item the section $s_2$ does not satisfy the equivalent conditions
    of Lemma \ref{degenlemm} on any irreducible component of $X$, and    
  \item there exists a Zariski-dense subset $S\subseteq X$
    such that $s_1(x)$ and $s_2(x)$ are linearly dependent in the
    fiber $\scr{F}(x)$ for all $x\in S$.
  \end{enumerate}
  Then there exists a global meromorphic function $\Phi$ on $X$ such
  that $s_1\otimes 1 = s_2\otimes \Phi$ in
  $\scr{F}\otimes_{\OO_X}\scr{M}_X$.  
\end{lemm}
\begin{proof}
  Let $\pi:\widetilde{X}\longrightarrow X$ denote the normalization of
  $X$.  By Lemma \ref{bclemm2},  it suffices to show
  that there exists a section $\Psi\in
  \Gamma(\widetilde{X},\scr{M}_{\widetilde{X}})$ such
  that $$\pi^*(s_1)\otimes 1 =
  \pi^*(s_2)\otimes\Psi\ \ \mathrm{in}\ \ \pi^*\scr{F}
  \otimes_{\OO_{\widetilde{X}}} \scr{M}_{\widetilde{X}},$$ which we may
  check on each connected component of $\widetilde{X}$ separately.
  Note that the hypotheses on $s_1$ and $s_2$ in the statement are
  satisfied on each such component by the pull-backs $\pi^*(s_1)$ and
  $\pi^*(s_2)$ since $\pi$ is finite.

   Thus we are reduced to proving the lemma in the case where the
   curve $X$ is moreover normal and connected.  In this case, the
   torsion-free coherent sheaf $\scr{F}^{\vee\vee}$ is locally-free.
   By Lemma \ref{mapsto0} applied to $\bigwedge_{\OO_X}^2\scr{F}$, the
   section $s_1\wedge s_2$ maps to zero
   in $$\left(\textstyle\bigwedge^{2}_{\OO_X}\scr{F}\right)\otimes_{\OO_X}\scr{M}_X
   \cong\textstyle\bigwedge^2_{\scr{M}_X}(\scr{F}\otimes_{\OO_X}\scr{M}_X).$$
   We will construct the meromorphic function $\Phi$ locally and glue.
   Let $\Sp(A)\subseteq X$ be an admissible affinoid open.  Then $A$
   is a domain, say with field of fractions $K$.  Choose a finite
   $A$-module $M$ with $\scr{F}\cong\widetilde{M}$ on $\Sp(A)$, so
   that
 $$\Gamma(\Sp(A),\scr{F}\otimes_{\OO_X}\scr{M}_X) \cong M\otimes_A K$$
   by Lemma \ref{bclemm1}.  The section $s_2\otimes 1$ restricts to a
   nonzero element of this finite-dimensional $K$-vector space by
   Lemma \ref{degenlemm}.  Since $(s_1\otimes 1)\wedge (s_2\otimes
   1)=0$ on $\Sp(A)$, there exists $\psi\in K$ such that $s_1\otimes 1
   = s_2\otimes \psi$.  That these $\psi$ glue to a global section of
   $\scr{M}_{X}$ follows from the non-vanishing of $s_2\otimes 1$ on a
   nonempty admissible open intersection of two admissible opens.
\end{proof}

\section{Some results of Waldspurger}\label{sec:wald}

In \cite{waldshimura} and \cite{waldspurger}, Waldspurger recasts the
Shimura lifting in terms of automorphic representations and uses this
perspective to prove some powerful theorems relating the coefficients
of half-integral weight modular forms to special values of the
$L$-functions of their integral weight Shimura lifts.  In this
section, we recall a key construction and two of the main results of
\cite{waldspurger}.  We will use the notation of \cite{waldspurger}
freely here, and we will take $N$ to be any positive integer
for the remainder of this section.

Let $\chi$ be a Dirichlet character modulo $4N$ and let $f_0$ be a
newform of weight $k-1$ and nebentypus $\chi^2$.  Following
Waldspurger, we let $\underline{\lambda}_\ell$ denote the eigenvalue
of $T_\ell$ or $U_\ell$ on $f_0$, and we set $$S_{k/2}(4N,\chi,f_0) =
\{F\in S_{k/2}(4N,\chi)\ |\ T_{\ell^2}F =
\underline{\lambda}_{\ell}F\ \mbox{for almost all}\ \ell\nmid 2N\}.$$
Here, $S_{k/2}(4N,\chi)$ is the space of cusp forms of level $4N$ and
weight $k/2$ with nebentypus character $\chi$ modulo $4N$.  In
\cite{waldspurger}, Waldspurger provides a recipe for building the
space $S_{k/2}(4N,\chi,f_0)$ from local data associated to $f_0$ and
$\chi$.  The details of this recipe are far too long to reproduce
here.  Nonetheless, we provide an outline that will allow the
motivated reader to compare the details of the arguments below to the
construction in \cite{waldspurger}.

For each prime number $\ell$, Waldspurger defines a certain
non-negative integer $\tilde{n}_\ell\leq \ord_\ell(4N)$.  Then, for
each integer $e$ with $\tilde{n}_\ell\leq e\leq \ord_{\ell}(4N)$, he
defines a certain set of functions $U_\ell(e)$ that we may take to be
defined on the set of positive integers.  The nature of
$\tilde{n}_\ell$ and the functions in $U_\ell(e)$ is dictated by local
(at $\ell$) data associated to the integral weight newform $f_0$ and
the half-integral weight nebentypus $\chi$.

For a positive integer $n$, let $n^\sqf$ denote the square-free part
of $n$, and let $A$ be a function defined on the set of square-free
positive integers.  Now, for each prime $\ell$, choose an integer $e$
as above and an element $c_\ell\in U_\ell(e)$, and from these data we
form the $q$-expansion $$f(A,\{c_v\}) = \sum_{n=1}^\infty
A(n^\sqf)\prod_v c_v(n)q^n,$$ where we take $c_\infty(n) =
n^{(k-2)/4}$.  The following is Th\'eor\`eme 1 of \cite{waldspurger}.
\begin{theo}[Waldspurger]\label{waldmain}
  Let $k\geq 5$ be an odd positive integer and suppose that $f_0$
  satisfies hypotheses (H1) and (H2). There exists a function $A$
  defined on the set of square-free positive integers satisfying the
  following two conditions:
  \begin{enumerate}
  \item[(a)] $A(n)^2 = L(f_0\otimes\chi_0^{-1}\chi_n,(k-1)/2)
    \varepsilon(\chi_0^{-1}\chi_n,1/2)\ $, and
  \item[(b)] the collection of $f(A,\{c_v\})$ defined above
    spans  the space $S_{k/2}(4N,\chi,f_0)$.
  \end{enumerate}
\end{theo}
\begin{rema}\label{remarks}\
  \begin{enumerate}
  \item Let $\pi=\otimes \pi_v$ be the automorphic representation
    attached to $f_0$.  Hypothesis (H1) is the assertion that, for all
    $\ell$ with $\pi_\ell$ irreducible principal series associated to
    the characters $\mu_1, \mu_2$ of $\Q_p^\times$, we have
    $\mu_1(-1)=\mu_2(-1)=1$.  By a theorem of Flicker, this is
    equivalent to $f_0$ being in the image of the Shimura lifting at
    some level.  We observe for later use that if the local central
    character $\mu_1\mu_2$ is even and the conductor of  $\pi_\ell$
    divides $\ell$, then hypothesis (H1) is automatically satisfied
    since at most one of $\mu_1$ and $\mu_2$ can be ramified.
  \item 
    Hypothesis (H2) pertains to $\pi_2$.  For our purposes, it
    suffices to note that (H2) is satisfied if $\pi_2$ is not
    supercuspidal.
  \item 
    While we have chosen to use the more classical normalization of the
    argument of the $L$-function, we have preserved Waldspurger's
    normalization for the $\varepsilon$ factor above.  This choice is
    irrelevant for our purposes, as all we use about this value is that
    it is non-vanishing.  
  \end{enumerate}
\end{rema}

\begin{coro}\label{wald1}
  With notation as in Theorem \ref{waldmain}, if $n$ is a square-free
  positive integer such that
  $$L(f_0\otimes \chi_0^{-1}\chi_n,(k-1)/2) = 0$$ then $a_n\equiv 0$
  on $S_{k/2}(4N,\chi,f_0)$.
\end{coro}

The converse of Corollary \ref{wald1} is false.  However, Theorem
\ref{waldmain} can be used to analyze the vanishing of $a_n$ in
general.  This analysis is carried out in in case $N$ is odd and
square-free in Proposition \ref{hardprop} in order to characterize in
Corollary \ref{components} the degenerate irreducible components (see
Definition \ref{def:degen}) for a given coefficient section $a_n$.

\begin{theo}[Waldspurger]\label{wald2}
  Let $n$ and $m$ be square-free positive integers such that $n/m\in
  (\Q_{\ell}^\times)^2$ for all $\ell\mid 2N$.  If $F\in
  S_{k/2}(4N,\chi,f_0)$ then
  \begin{eqnarray*}
    \lefteqn{ a_n(F)^2 L(f_0\otimes \chi_0^{-1}\chi_m, (k-1)/2)
      \chi(m/n) m^{k/2-1} } && \\ && = a_m(F)^2 L(f_0\otimes
    \chi_0^{-1}\chi_n, (k-1)/2) n^{k/2-1}
    \end{eqnarray*}
\end{theo}
\begin{proof}
  If $F=0$ then the statement is trivial. Otherwise hypothesis (H1) is
  satisfied by Proposition 2 of \cite{waldspurger}, and this result is
  Corollaire 2 of \cite{waldspurger} applied to $F$.
\end{proof}

\section{Interpolation of square roots}\label{sec:interp}

In this section, we establish the basic interpolation result on the
half-integral weight side.  Here, we make no assumption on the tame
level. The trade-off is that we will have little control over which
irreducible components we may interpolate.  In the next section, we
impose the square-free condition and determine these components, as
well as move the interpolation to the integral weight eigencurve.

Fix a a primitive $n^{\mathrm{th}}$ root of unity $\zeta_{4Np}\in\C_p$
thought of as a point on the Tate curve $\Tate(q)$.  As defined in
Section 5 of \cite{hieigencurve}, half-integral weight modular forms
and families thereof have $q$-expansions associated to this point on
$\Tate(q)$.  For a fixed positive integer $n$, the $n^\mathrm{th}$
coefficient in this expansion gives an element of the dual of the
Banach space of families of forms.  It is a simple matter to verify
that these elements glue to a section $$a_n\in
\Gamma(\widetilde{D}(4N), \scr{N}^*),$$ and that the image of this
element in the fiber $\scr{N}^*(x)$ is the $n^\mathrm{th}$ coefficient
map on the eigenspace associated to $x$ under the identification
(\ref{fibers}).

\begin{defi}\label{def:degen}
  We will say that an $a_n$ is \emph{degenerate} on an irreducible
  component $C$ of $\widetilde{D}(4N)_\red$ if $a_n$ satisfies the
  equivalent conditions of Lemma \ref{degenlemm} on $C$.
\end{defi}

\begin{prop}\label{lindepapp}
  Let $N$ be a positive integer that is relatively prime to $p$ and
  let $n$ and $m$ be square-free positive integers such that $n/m\in
  (\Q_\ell^\times)^2$ for all $\ell | 2Np$.  Let $C$ be an irreducible
  component of $\widetilde{D}(8N)_\red$, and assume that $C$ is
  generically principal series or Steinberg at the prime $2$.  Then,
  there exists a Zariski-dense subset $S$ of $C$ such that $a_n(x)$
  and $a_m(x)$ are linearly dependent in the fiber $\scr{N}^*(x)$ for
  all $x\in S$.
\end{prop}
\begin{proof}
  The set of low-slope (and hence strictly classical) points is
  Zariski-dense in $C$.  By hypothesis, the same is then true of the
  set of low-slope points that are not supercuspidal at the prime $2$.
  Let $x$ be such a point of weight
  $\kappa(t)=t^{(k-1)/2}\kappa'(t)$ with $k$ odd and at least $5$ and
  $\kappa'$ of conductor $p^r$.  Let $\lambda_x$ be the corresponding
  system of eigenvalues and let $S_x$ denote the
  eigenspace of forms for $\lambda_x$.  In particular
  $S_x\subseteq
  S_{k/2}(8Np^r,\chi\kappa',f_0)$ where $f_0$ denotes the
  integral weight newform associated to $x$.

  If $$L(f_0\otimes \chi_0^{-1}\kappa'^{-1}\chi_n,(k-1)/2)=0$$ then
  $a_n\equiv 0$ on $S_{x}$ by Corollary \ref{wald1}.  Thus
  $a_n(x)=0$, so $a_n(x)$ and $a_m(x)$ are trivially linearly
  dependent.  On the other hand, suppose that $$L(f_0\otimes
  \chi_0^{-1}\kappa'^{-1}\chi_n,(k-1)/2)\neq 0.$$ By Lemma
  \ref{wald2}, if $F\in S_{x}$ and $a_n(F)=0$, then
  $a_m(F)=0$ as well.  Thus the kernel of the linear functional
  $a_n(x)$ is contained in that of $a_m(x)$, and again we see that
  these two are linearly dependent.
\end{proof}

\begin{coro}\label{phiexists}
  With notation as in Proposition \ref{lindepapp}, suppose further
  that $a_n$ is not degenerate on $C$.  Then there exists a meromorphic
  function $\Phi_{m,n}$ on $C$ with the property that $a_m\otimes 1 =
  a_n\otimes \Phi_{m,n}$ holds in $\scr{N}_\red^*\otimes \scr{M}_C$.
\end{coro}
\begin{proof}
  This follows from the Proposition \ref{lindepapp} and Lemma
  \ref{lindep}.
\end{proof}

By \S1.2 of \cite{conradirredcpnts}, the singular (equivalently,
non-normal) locus in the reduced component $C$ is a proper analytic
set.

\begin{theo}\label{maininterp}
  With notation as in Proposition \ref{lindepapp}, suppose that $x\in
  C$ is a point of weight $\kappa(t) =t^{(k-1)/2}\kappa'(t)$, with
  $k\geq 5$ and $\kappa'$ torsion, satisfying
  \begin{itemize}
  \item $x$ is strictly classical,
  \item $x$ is not supercuspidal at $2$,
  \item $C$ is smooth at $x$,
  \item $\scr{N}^*_\red$ is torsion-free on $C$ at $x$, and
  \item $a_n(x)\neq 0$.
  \end{itemize}
  Then $\Phi_{m,n}$ is regular at $x$ and we have
  $$\Phi_{m,n}(x)^2 = \frac{L(f_0\otimes
    \chi_0^{-1}\kappa'^{-1}\chi_m,(k-1)/2)} {L(f_0\otimes
    \chi_0^{-1}\kappa'^{-1}\chi_n,(k-1)/2)}\chi(m/n)\kappa(m/n)(m/n)^{-1/2}.$$
\end{theo}
\begin{rema}
  The smoothness assumption is likely known to be superfluous at
  classical points, but the author knows of no reference in which this
  is completely proven for the eigencurves used here.
\end{rema}
\begin{proof}
  Let $x$ be as in the statement and choose a normal admissible open
  affinoid $\Sp(A)\subseteq C$ containing $x$.  Choose a finite
  $A$-module $M$ with $\scr{N}_\red^*\cong \widetilde{M}$ on $\Sp(A)$
  and let $K$ denote the fraction field of $A$.  Restricting to
  $\Sp(A)$, we may regard $a_n$ and $a_m$ as elements of $M$, and
  $\Phi_{m,n}$ as an element of $K$.  In particular, we
  have $$a_m\otimes 1 = a_n \otimes \Phi_{m,n}$$ in $M\otimes_A K$.

  Write $\Phi_{m,n}=f/g$ for $f,g\in A$ and clear denominators to
  get $$(ga_m-fa_n)\otimes 1=0.$$ Let $\mathfrak{m}$ denote the
  maximal ideal of $A$ corresponding to the point $x$.  By faithful
  flatness of completion, $M$ is torsion-free at $x$, so the equality
  $ga_m = fa_n$ actually holds in the localization $M_\mathfrak{m}$.
  By normality, $A_\mathfrak{m}$ is a DVR, so we may write
  $f=\pi^\alpha u$ and $g=\pi^\beta v$ where $\pi$ is a uniformizer
  for $A_\mathfrak{m}$ and $u$ and $v$ are units in this local ring.
  If $\beta>\alpha$, then we have $$a_n =
  vu^{-1}\pi^{\beta-\alpha}a_m\in \pi M_\mathfrak{m},$$ contrary to
  the hypothesis that $a_n(x)\neq 0$.  Thus $\beta\leq \alpha$, and
  $\Phi_{m,n}$ is regular at $x$.

  Since $a_n(x)\neq 0$, there exists a nonzero classical form $F$ in
  the eigenspace corresponding to $x$ with the property that
  $a_n(F)\neq 0$.  By the previous paragraph, we may write $$a_m(F) =
  a_n(F)\Phi_{m,n}(x).$$ By Corollary \ref{wald1} we have
  $L(f_0\otimes \chi_0^{-1}\kappa'^{-1}\chi_n, (k-1)/2)\neq 0$, and
  Theorem \ref{wald2} gives
  \begin{eqnarray*}
    \Phi_{m,n}(x)^2 &=& \frac{a_m(F)^2}{a_n(F)^2} \\ &=&
    \frac{L(f_0\otimes \chi_0^{-1}\kappa'^{-1}\chi_m,(k-1)/2)}
         {L(f_0\otimes
           \chi_0^{-1}\kappa'^{-1}\chi_n,(k-1)/2)}\chi(m/n)
         \kappa'(m/n)(m/n)^{(k-2)/2} \\ &=&
    \frac{L(f_0\otimes \chi_0^{-1}\kappa'^{-1}\chi_m,(k-1)/2)}
         {L(f_0\otimes
           \chi_0^{-1}\kappa'^{-1}\chi_n,(k-1)/2)}\chi(m/n)
         \kappa(m/n)(m/n)^{-1/2}.
  \end{eqnarray*}
  \end{proof}

\section{Characterization of the degenerate components}\label{sec:degen}

Let $N$ be a square-free positive integer that is relatively prime to
$2p$.  In this section we determine necessary and sufficient
conditions under which $a_n(x)=0$ for strictly classical points
$x\in\widetilde{D}(8N)$ whose Shimura lifting is not supercuspidal at
the prime $2$.  We then use these conditions to characterize the
components on which $a_n$ is degenerate.  Suppose that $x$ is a
strictly classical point of weight $\kappa(t) = t^{(k-1)/2}\kappa'(t)$
with $k\geq 5$ and $\kappa'$ of conductor $p^r$ and tame nebentypus
$\chi$ modulo $8N$.  Let $f_0$ denote the newform associated to this
point and observe that the eigenspace $S_x$ corresponding to $x$
satisfies $$S_x\subseteq S_{k/2}(8Np^r,\chi\kappa',f_0)$$ since $x$ is
strictly classical.

\begin{prop}\label{hardprop}\
With notation as above, let $\pi = \otimes_v\pi_v$ denote the
automorphic representation generated by the integral weight newform
$f_0$ associated to $x$, and suppose that $\pi_2$ has conductor at
most $2$.  The fiber $\scr{N}^*(x)$ has dimension
  $$\dim\ \!\scr{N}^*(x) =  \prod_{\ell\mid 2N}d_\ell$$ where
$$d_\ell = \begin{cases} 1+v_2(\ell) & \pi_\ell\ ramified \\2+v_2(\ell) & \pi_\ell\ unramified\end{cases}$$
Furthermore, $a_n(x)=0$ if and only if one of the following holds at
$x$:
    \begin{enumerate}
      \renewcommand{\labelenumi}{({\it \roman{enumi}})}
    \item $L(f_0\otimes \chi_0^{-1}\kappa'^{-1}\chi_n,(k-1)/2) = 0$
    \item for some $\ell\mid 2N$, $\pi_\ell$ is Steinberg
      with $$\underline{\lambda}_\ell = (\ell,n)_\ell\chi_0^{(\sim
        \ell)}(\ell)\kappa(\ell)$$ and
      \begin{itemize}
      \item[$\bullet$] if $\ell=2$, then $(n,-1)_2=\chi_0^{(2)}(-1)$
        and either $2\nmid n$ and $\chi_0\equiv 1$ on $1+4\Z_2$ or
        $2\mid n$ and $\chi_0\not\equiv 1$ on $1+4\Z_2$.
      \item[$\bullet$] if $\ell|N$, then either $\ell\nmid n$ and
        $\chi$ is unramified at $\ell$ or $\ell\mid n$ and $\chi$ is
        ramified at $\ell$
      \end{itemize}
    \item $\pi_p$ is Steinberg with $$\underline{\lambda}_p =
      \lambda_x(U_{p^2}) = (p,n)_p\chi_0(p)p^{(k-3)/2}$$ and either
      $p\nmid n$ and $\kappa'$ is unramified at $p$ (which is to say
      $\kappa'=1$) or $p\mid n$ and $\kappa'$ is ramified at $p$.
    \end{enumerate}
\end{prop}
\begin{proof}
  Let $\kappa(t) = t^{(k-1)/2}\kappa'(t)$ denote the weight of $x$.
  In addition to using the notation in the preceding paragraphs, we
  borrow the following notation from \cite{waldspurger} for the
  duration of the proof: for a prime number $\ell$,
  $\underline{\lambda}_\ell$ will continue to denote the Hecke
  eigenvalue of $f_0$ for the appropriate operator (either $T_\ell$ or
  $U_\ell$) and we let $\lambda_\ell =
  \ell^{1-k/2}\underline{\lambda}_\ell$.  If $\ell\nmid 2Np$, we
  denote by $\alpha_\ell$ and $\alpha_\ell'$ the roots of the
  polynomial $$X^2 - \lambda_\ell X + (\chi(\ell)\kappa'(\ell))^2.$$
  For convenience, we also introduce the notation
  $\underline{\alpha}_\ell = \ell^{k/2-1}\alpha_\ell$.
  
  The following table, compiled from Section VIII of
  \cite{waldspurger}, lists the elements that arise in the sets
  $U_\ell(e)$ from the construction outlined in Section \ref{sec:wald}
  for the local conditions that can arise with our restrictions on the
  tame level.  The notation is that of Waldspurger, who gives explicit
  formulas for all of the functions below. We have not listed the
  value of $\tilde{n}_\ell$ and the particular integers $e$ that give
  rise to these elements, as these will be of no use to us here.
  \begin{center}
  \begin{tabular}{|c|c|c|}\hline
    $\ell$ & local condition & $c_\ell\in U_\ell(e)$ \\\hline \hline
    $\ell\nmid 2Np$ & none& $c_\ell^0[\lambda_\ell]$ \\\hline
    \multirow{7}{*}{$\ell=2$} & $\pi_2$ unramified, $\chi_0\equiv 1$
    on $1+4\Z_2$, $\alpha_2\neq \alpha_2'$ & $c_2'[\alpha_2],
    c_2'[\alpha_2'], \gamma^0[0]$ \\ & $\pi_2$ unramified,
    $\chi_0\equiv 1$ on $1+4\Z_2$, $\alpha_2=\alpha_2'$ &
    $c_2'[\alpha_2], c_2''[\alpha_2], \gamma^0[0]$ \\ & $\pi_2$
    unramified, $\chi_0\not\equiv 1$ on $1+4\Z_2$, $\alpha_2\neq
    \alpha_2'$ & $'\!c_2[\alpha_2], '\!c_2[\alpha_2'], \gamma''[0]$
    \\ & $\pi_2$ unramified, $\chi_0\not\equiv 1$ on $1+4\Z_2$,
    $\alpha_2=\alpha_2'$ & $'\!c_2[\alpha_2], ''\!c_2[\alpha_2],
    \gamma''[0]$ \\ & $\pi_2$ ramified principal series &
    $c_2^*[\sqrt{\lambda_2}], \gamma''[0]$ \\ & $\pi_2$ Steinberg,
    $\chi_0\equiv 1$ on $1+4\Z_2$ & $c_2^s[\lambda_2], \gamma^0[0]$
    \\ & $\pi_2$ Steinberg, $\chi_0\not\equiv 1$ on $1+4\Z_2$ &
    $^s\!c_2[\lambda_2], \gamma''[0]$ \\ \hline
    \multirow{6}{*}{$\ell\mid Np$} & $\pi_\ell$ unramified,
    $\chi\kappa'$ unramified & $c_\ell^0[\lambda_\ell],
    c_\ell'[\alpha_\ell]$ \\ & $\pi_\ell$ unramified, $\chi\kappa'$
    ramified, $\alpha_\ell\neq \alpha_\ell'$ &
    $'\!c_\ell[\alpha_\ell], '\!c_\ell[\alpha_\ell']$\\ & $\pi_\ell$
    unramified, $\chi\kappa'$ ramified, $\alpha_\ell= \alpha_\ell'$ &
    $'\!c_\ell[\alpha_\ell], ''\!c_\ell[\alpha_\ell]$ \\ & $\pi_\ell$
    ramified principal series & $c_\ell^*[\sqrt{\lambda_\ell}]$ \\ &
    $\pi_\ell$ Steinberg, $\chi\kappa'$ unramified &
    $c_\ell^s[\lambda_\ell]$ \\ & $\pi_\ell$ Steinberg, $\chi\kappa'$
    ramified & $^s\!c_\ell[\lambda_\ell]$ \\\hline
  \end{tabular}
  \end{center}

  As outlined in Section \ref{sec:wald}, Waldspurger uses these local
  functions to construct the space $S_{k/2}(8Np^r,\chi\kappa',f_0)$.
  From the explicit formulas for the $c_\ell$ in \cite{waldspurger},
  the effect of the appropriate Hecke operator at $\ell$ on the form
  $f(A,\{c_v\})$ can be read off from the element $c_\ell$.  In our
  setting, this operator is $T_{\ell^2}$ only in the case $c_\ell =
  c^0_\ell[\lambda_\ell]$, and a form constructed using this element is
  a $T_{\ell^2}$-eigenform with eigenvalue $\underline{\lambda}_\ell$.
  In all other cases, the relevant Hecke operator is $U_{\ell^2}$.
  With the exception of the functions $c_\ell^*[\sqrt{\lambda_\ell}]$,
  $c_\ell''[\alpha_\ell]$, and $''\!c_\ell[\alpha_\ell]$, a function
  constructed from an element in the above table is a
  $U_{\ell^2}$-eigenform whose eigenvalue is the argument of $c_\ell$
  in brackets times $\ell^{k/2-1}$ (i.e. the ``underlined version'' of
  this argument).  In the case $c_\ell^*[\sqrt{\lambda_\ell}]$, the
  same is true except the eigenvalue is $\underline{\lambda}_\ell$.
  If $c_\ell = c_\ell''[\alpha_\ell]$ then we
  have $$U_{\ell^2}f(A,\{c_v\}) = \underline{\alpha}_\ell f(A,\{c_v\})
  + \underline{\alpha}_\ell f(A,\{b_v\})$$ where $$b_v = \begin{cases}
    c_v & v\neq \ell \\ c_\ell'[\alpha_\ell] & v=\ell\end{cases}$$ The
    behavior in the $''\!c_\ell$ case is the same as this case with
    all the primes switched to the left side.

    Since $c_\ell = c_\ell^0[\lambda_\ell]$ for $\ell\nmid 4Np$, forms
    in $S_{k/2}(8Np^r,\chi\kappa',f_0)$ satisfy $T_{\ell^2}F =
    \underline{\lambda}_\ell F$ for \emph{all} $\ell\nmid 4Np$.  Thus,
    the subspace $S_x$ is obtained by further requiring only that the
    $F\in S_{k/2}(8Np^r,\chi\kappa',f_0)$ also satisfy
    $U_{p^2}F=\lambda_x(U_{p^2})F$.  In the ramified cases where there
    is just one $c_p$ in the above table, we have $\lambda_x(U_{p^2})
    = \underline{\lambda}_p$, and this condition is automatic.  In the
    unramified cases where there are two functions, this condition
    picks out exactly one of the functions, namely, $c_p'[\alpha_p]$
    or $'\!c_p[\alpha_p]$ depending on whether $\kappa'$ is unramified
    or ramified, respectively.  Here, in case $\alpha_p\neq
    \alpha_p'$, the choice of $\alpha_p$ is dictated by the value of
    $\lambda_x(U_{p^2})$.

    Exploiting these properties of the Hecke operators acting on the
    forms $f(A,\{c_v\})$, it is a simple exercise to show that, for a
    fixed $f_0$ in our setting, the collection of
    forms $f(A,\{c_v\})$ constructed from the elements in the above
    table are linearly dependent if an only if one of them vanishes
    identically.  That the latter does not occur follows from
    well-known results about non-vanishing of quadratic twists of
    central modular $L$-values (see, for example, \cite{bfh}).  This
    linear independence establishes the dimension claim of the
    proposition.
    
    By Theorem \ref{waldmain}, for a square-free positive integer $n$,
    we have $a_n(x)=0$ if and only if the $L$-value in Theorem
    \ref{waldmain} vanishes, or, for some $\ell$, we have
    $c_{\ell}(n)=0$ for all local functions appearing in the above
    table at $\ell|2Np$ that occur on the subspace $S_x$.  Again
    examining the explicit formulas in \cite{waldspurger} for the
    $c_\ell$ listed in this table (and using the fact that
    $\lambda_\ell\neq 0$ for $\ell$ dividing the square-free conductor
    of $f_0$) one concludes that, for $\ell\mid 2Np$, all
    $c_{\ell}(n)$ occurring in $S_x$ vanish if and only if $\pi_\ell$
    is Steinberg with $$\underline{\lambda}_\ell =
    (\ell,n)_\ell(\chi_0\kappa')^{(\sim \ell)}(\ell)\ell^{(\ell-3)/2}
    = \begin{cases} (\ell,n)_\ell\chi_0^{(\sim
        \ell)}(\ell)\ell^{-1}\kappa(\ell) & \ell\neq p
      \\ (p,n)_p\chi_0(p)p^{(k-3)/2} & \ell=p\end{cases}$$ and
  \begin{itemize}
  \item[$\bullet$] if $\ell=2$, then we have
    $(n,-1)_2=\chi_0^{(2)}(-1)$ and either $2\nmid n$ and
    $\chi_0\equiv 1$ on $1+4\Z_2$ or $2\mid n$ and $\chi_0\not\equiv
    1$ on $1+4\Z_2$
  \item[$\bullet$] if $\ell|Np$ then either $\ell\nmid n$ and
    $\chi\kappa'$ is unramified at $\ell$ or $\ell\mid n$ and
    $\chi\kappa'$ is ramified at $\ell$.
  \end{itemize}
\end{proof}

\begin{coro}\label{components}
  Let $C$ be an irreducible component of $\widetilde{D}(8N)_\red$ and
  suppose that $C$ maps to a component of $D(2N)_\red$ under the Shimura
  lifting.  Then the section $a_n$ is degenerate on $C$ if and only if
  one of the following holds:
  \begin{enumerate}
    \renewcommand{\labelenumi}{(\alph{enumi})}
  \item    
    $C$ contains a Zariski-dense set of classical points $x$ at which
    $$L(f_0\otimes\chi_0^{-1}\kappa'^{-1}\chi_n,(k-1)/2)=0$$
  \item for some $\ell\mid 2N$, $C$ is Steinberg at $\ell$ of sign
    $(\ell,n)_\ell$ and one of the following holds
    \begin{itemize}
    \item[$\bullet$] $\ell=2$, $(n,-1)_2=\chi_0^{(2)}(-1)$,
      $2\nmid n$, and $\chi_0\equiv 1$ on $1+4\Z_2$
    \item[$\bullet$] $\ell=2$, $(n,-1)_2=\chi_0^{(2)}(-1)$,
      $2\mid n$, and $\chi_0\not\equiv 1$ on $1+4\Z_2$
     \item[$\bullet$] $\ell\mid N$, $\ell\nmid n$, and $\chi$ is
       unramified at $\ell$
     \item[$\bullet$] $\ell\mid N$, $\ell\mid n$, and $\chi$ is
       ramified at $\ell$.
    \end{itemize}
  \end{enumerate}
\end{coro}
\begin{proof}
  The ``if'' portion is clear from Proposition \ref{hardprop}.
  Suppose that $a_n$ is degenerate on $C$ and condition (b) in the
  statement does not hold.  The collection of points that either
  satisfy condition ({\it ii}) of Proposition \ref{hardprop} or at
  which $\pi_p$ is Steinberg lies in a proper analytic set in $C$.  By
  the results of Section 4 of \cite{rigidshimura}, the collection of
  low-slope (hence strictly classical) points is Zariski-dense in the
  complement of this proper analytic set in $C$. By Proposition
  \ref{hardprop}, we must have
  $L(f_0\otimes\chi_0^{-1}\kappa'^{-1}\chi_n,(k-1)/2) =0$ at these
  points.
\end{proof}

\begin{coro}\label{torsionfree}
  Suppose that $x\in \widetilde{D}(8N)_\red$ is a strictly classical
  point lying on an irreducible component $C$ that maps to a component
  of $D(2N)_\red$ under the Shimura lifting.  The sheaf
  $\scr{N}^*_\red$ is torsion-free on $C$ at $x$.
\end{coro}
\begin{proof}
  Let $x$ be as in the statement.  The local ring of $C$ at $x$ is a
  domain, so it suffices to show that the fiber rank and generic rank
  of $\scr{N}^*$ coincide at $x$.  By hypothesis, $C$ is either
  generically irreducible principal series or generically Steinberg at
  each $\ell\mid 2N$.  As we have observed, if $C$ is generically
  ramified principal series or Steinberg then it is entirely so, while
  if it is generically unramified principal series, then $C$ may
  contain a proper analytic set of Steinberg points.
  
  By Proposition \ref{hardprop}, the generic rank of $\scr{N}^*$ on
  $C$ is $2^r\cdot s$ where $r$ is the number of $\ell\mid N$ at which
  $C$ is generically unramified principal series and $s=2$ or $3$
  according to whether $C$ is generically ramified (i.e. ramified
  principal series or Steinberg) or generically unramified principal
  series at the prime $2$, respectively.  But by Proposition
  \ref{hardprop} and the previous paragraph (which in effect says that
  ``ramified does not degenerate to unramified''), the dimension of
  the fiber does not increase at a strictly classical point.  By
  upper-semicontinuity of fiber dimension, this dimension cannot
  decrease from the generic rank either.  Thus the generic rank and
  fiber rank coincide at smooth strictly classical points of
  $C$.\end{proof}

In the following Theorem, we move the interpolation to the integral
weight eigencurve via the Shimura lifting.  In so doing, we may
effectively eliminate components that are degenerate for reasons other
than the vanishing of $L$-values by choosing a component on the
half-integral weight side with some care.

\begin{theo}
  Let $N$ be a square-free positive integer that is relatively prime
  to $2p$ and let $n$ and $m$ be square-free positive integers with
  $m/n\in (\Q_\ell^\times)^2$ for all $\ell\mid 2Np$.  Fix a square
  Dirichlet character $\psi$ modulo $2N$.  There exists a character
  $\chi$ modulo $8N$ such that $\chi^2=\psi$ with the following
  property: for any irreducible component $C$ of $D(2N)_\red$ of even
  weight and tame nebentypus $\psi$ on which $L(f_0\otimes
  \chi_0^{-1}\kappa'^{-1}\chi_n,(k-1)/2)$ does not vanish generically,
  there exists a meromorphic function $\Phi_{m,n}$ on $C$ such that
   $$\Phi_{m,n}(x)^2 = \frac{L(f_0\otimes
    \chi_0^{-1}\kappa'^{-1}\chi_m,(k-1)/2)} {L(f_0\otimes
    \chi_0^{-1}\kappa'^{-1}\chi_n,(k-1)/2)}\chi(m/n)\kappa(m/n)
  (m/n)^{-1/2}$$ holds for strictly classical points $x\in C$ outside
  of the discrete set at which $C$ is singular, condition (iii) of
  Proposition \ref{hardprop} is satisfied, or the denominator
  vanishes.
\end{theo}
\begin{proof}
  Let $\ell\mid 2N$ be a prime at which $\psi$ is unramified.
  Utilizing quadratic characters of prime conductor, we may choose
  $\chi$ to satisfy the following conditions at each such $\ell$:
  \begin{itemize}
  \item[$\bullet$] for $\ell\neq 2$, $\chi$ is unramified at $\ell$ if
    and only if $\ell\mid n$
  \item[$\bullet$] at $\ell=2$, $\chi_0^{(2)}\equiv 1$ on $1+4\Z_2$ if
    an only if $2\mid n$
  \end{itemize}
 We remark that it is in satisfying the second of these conditions for
 odd $n$ that necessitates working at level $8N$ instead of $4N$.

  Let $C$ be as in the statement and let $x_0\in C$ be a classical
  point of weight character $t^{k-1}\kappa'(t)^2$ where $\kappa'$ is a
  character modulo $p^r$. The even hypothesis combined with
  finite-slope and conductor considerations ensure that hypothesis
  (H1) is satisfied (see Remark \ref{remarks}(1)).  Suppose that $x_0$
  lies on at most one component of the image of $\widetilde{D}(8N)$ in
  $D(4N)$ under the Shimura lifting.  Since $\pi_2$ is not
  supercuspidal at $x_0$, hypothesis (H2) is satisfied, and the
  construction of \cite{waldspurger} outlined in Section
  \ref{sec:wald} implies that $S_{k/2}(8Np^r,\chi\kappa',f_0(x_0))\neq
  0$ and that this space moreover contains a nonzero eigenform for
  $U_{p^2}$ with eigenvalue $\lambda_{x_0}(U_p)$.  This form gives
  rise to a point on $\tilde{x}_0\in \widetilde{D}(8N)$ that maps to
  $x_0$ under the Shimura lifting.  Let $\widetilde{C}$ denote the
  unique irreducible component of $\widetilde{D}(8N)_\red$ containing
  this point (so that, in particular, $\widetilde{C}$ maps
  isomorphically to $C$ under the Shimura lifting).  Since the tame
  nebentypus is constant on $\widetilde{C}$, condition (b) of
  Corollary \ref{components} does not hold on $\widetilde{C}$ by our
  choice of $\chi$.  By hypothesis, condition (a) does not hold
  either, and we conclude that $a_n$ is non-degenerate on
  $\widetilde{C}$.

  Let $\Phi_{m,n}$ be the meromorphic function on $\widetilde{C}$
  given by Corollary \ref{phiexists}.  Let $x\in C$ be a smooth point
  of low slope.  Then, $x$ is strictly classical regarded as a point
  on $\widetilde{C}$ under the Shimura lifting.  By Corollary
  \ref{torsionfree}, $\scr{N}^*_\red$ is torsion-free at $x$.  If
  neither condition (i) nor condition (iii) of Proposition
  \ref{hardprop} hold at $x$, then $a_n(x)\neq 0$ and Theorem
  \ref{maininterp} implies that $\Phi_{m,n}$ is regular at $x$ and its
  square has the claimed value.
\end{proof}

This finishes the proof of Theorem \ref{maintheorem} in the case of
even tame level.  If $N$ is odd, the result follows simply by pulling
back through the natural map $D(N)\longrightarrow D(2N)$.

\providecommand{\bysame}{\leavevmode ---\ }
\providecommand{\og}{``}
\providecommand{\fg}{''}
\providecommand{\smfandname}{\&}
\providecommand{\smfedsname}{\'eds.}
\providecommand{\smfedname}{\'ed.}
\providecommand{\smfmastersthesisname}{M\'emoire}
\providecommand{\smfphdthesisname}{Th\`ese}

\bibliographystyle{smfplain}

\end{document}